\documentclass[11pt,reqno,twoside,intlimits]{article}
\usepackage{indentfirst}
\usepackage[a4paper,outer=3cm,inner=3cm]{geometry}
\usepackage{fancyhdr}
\usepackage[intlimits,reqno]{amsmath}
\usepackage{amsthm, graphics}
\usepackage{amssymb}
\usepackage{verbatim}
\include{diagxy}
\input{xy}
\xyoption{all}
\usepackage{enumerate}

\usepackage{bbm}

\renewcommand{\S}{\mathcal{S}}

\newcommand{\be}{\begin{equation}}
\newcommand{\ee}{\end{equation}}
\newcommand{\ben}{\begin{enumerate}}
\newcommand{\een}{\end{enumerate}}

\renewcommand{\to}{\rightarrow}

\newcommand{\ol}{\overline}
\DeclareMathOperator{\Ad}{Ad}

\newtheorem{thm}{Theorem}
\newtheorem{cor}[thm]{Corrolary}

\newtheorem{prop}[thm]{Proposition}

\theoremstyle{definition}
\newtheorem{defn}{Definition}

\numberwithin{example}{section}
\theoremstyle{remark}

\usepackage{mathrsfs}

\newcommand{\co}{:}

\newcommand{\act}{\mathbin{\hbox{$<\kern-.4em\mapstochar\kern.4em$}}}
\newcommand{\ract}{\mathbin{\hbox{$\mapstochar\kern-.3em>$}}}

\usepackage{xspace}

\newcommand{\id}{\text{\upshape id}}

\usepackage[pointlessenum]{paralist}

\usepackage{subfig}
\usepackage{caption}
\usepackage{xcolor}

\begin{document}
\title{Symmetries of the space of connections on a principal $G$-bundle and related symplectic structures}

\author{Grzegorz Jakimowicz, Anatol Odzijewicz, Aneta Sli\.{z}ewska}

\maketitle
\tableofcontents
\begin{abstract} We investigate $G$-invariant symplectic structures on the cotangent bundle $T^*P$ of a principal $G$-bundle $P(M,G)$which are canonically related to  automorphisms of the tangent bundle $TP$ covering the identity map of $P$ and commuting with the action of $TG$ on $TP$. The  symplectic structures corresponding to connections on $P(M,G)$ are also investigated. The Marsden-Weinstein reduction procedure  for these symplectic structures is discussed.
\end{abstract}
\section{Introduction}

The phase space of a typical Hamiltonian system is the cotangent bundle $T^*P$ of its configurations space $P$ equiped with the standard symplectic form. Usually one considers  the case when Hamiltonian of this system is invariant with respect to the cotangent lift of an action of some group $G$ on $P$. Therefore, it is sometimes reasonable to replace the standard symplectic form with another  $G$-invariant symplectic form on $T^*P$  which retains certain properties. Assuming that the action of $G$ on $P$  is free and the quotient space $P/G$ is a manifold $M$ one can consider $P$ as the total space of the principal $G$-bundle $P(M,G)$.  Motivated  by the above we  consider here the following structures related to $P(M,G)$ in a natural way. 
\begin{enumerate}

\item[(i)] The space $CanT^*P$ of fibre-wise linear non-singular differential one-forms $\gamma$ on the cotangent bundle $T^*P$, which annihilate the vectors tangent to the fibres of $T^*P$. 

	\item[(ii)]The space $Conn P(M,G)$ of connections on the principal $G$-bundle $P(M,G)$.

 \item[(iii)] The group $Aut_{TG}TP$ of automorphisms  of the tangent bundle $TP$ covering the identity map of $P$ which commute with the action of the tangent group $TG$ on $TP$. 
\end{enumerate}

Since the canonical one-form $\gamma_0$ belongs to $CanT^*P$ and other elements of $CanT^*P$ posses some properties of $\gamma_0$, hence they are called generalized canonical forms  here.

There are  important  relations between  the above structures. Namely the group $Aut_{TG}TP$ acts on the both subspaces mentioned above. The action on $ConnP(M,G)$ is transitive and the action on $CanT^*P$ is free. The orbit $Can_{TG}T^*P$ of $Aut_{TG}TP$ through the canonical form $\gamma_0$ consists of $G$-invariant  generalized canonical forms $\gamma$ such that the momentum maps corresponding to symplectic  forms $\omega=d\gamma$ coincide with the  momentum map which corresponds to the standard symplectic form  $\omega_0=d\gamma_0$. In this way one obtains a family of symplectic forms $\omega_A$ on $T^*P$ enumerated by the elements $A$ of the group $Aut_{TG}TP$.

The group  $Aut_{TG}TP$ is investigated in Section \ref{sec2}, see Proposition \ref{prop:2}.

In Section \ref{sec3} we describe the action of the group $Aut_{TG}TP$ on $Conn P(M,G)$ and show that  $Aut_{TG}TP$ can be defined equivalently  as the group of symmetries of $Conn P(M,G)$, see Proposition \ref{prop:4}.

 The relations between the group  $Aut_{TG}TP$ and the space $Can_{TG}T^*P$ are investigated in Section  \ref{sec4}, see Proposition \ref{prop:5} and Proposition \ref{prop6}. Among other things it is shown that fixing a reference connection $\alpha$ one embeds the space of all connections into the space $Can_{TG}T^*P$ of generalized canonical  forms, see Corrolary \ref{cor:7}. The choice of the reference connection $\alpha$ also allows ones to define  a $G$-equivariant diffeomorphism $I_\alpha:T^*P\to \ol P\times T^*_eG$, see \cite{Mont,Weinstein:1977},   where $\ol P$ is the total space of the pull back $\ol P(T^*M,G)$ of $P(M,G)$ on $T^*M$ and $T^*_eG$ is the dual of the Lie algebra $T_eG$. The generalized canonical forms and related symplectic forms $\omega_A$ written on $\ol P\times T^*_eG$ obtain the form consistent with the structure of the bundle $\ol P\times T^*_eG\to T^*M$, see (\ref{86}) and (\ref{sympl}).

In Section  \ref{sec5} we discuss the $G$-Hamiltonian system on $(T^*P,\omega_A,J_0)$ which could be considered as a natural generalization of the ones investigated in \cite{Mont, Sternberg,Weinstein:1977}. The Marsden-Weinstein reduction procedure is applied to these systems.

\section{Symmetries of the tangent bundle  of a principal bundle }\label{sec2}

Let $P(M,G)$ be a principal bundle over a manifold $M$. Throughout  the paper we will denote the right action of the Lie group $G$ on $P$ by $\kappa:P\times G\to P$  and write $\mu\co P\to M$ for the bundle projection. We will use also the shorter notation $pg:=\kappa(p,g)$.
For a fixed $p\in P$ and $g\in G$ one has the corresponding maps $\kappa_p:G\to P$ and $\kappa_g:P\to P$ defined by
$$\label{kapp}\kappa_p:=\kappa(p,\cdot)\qquad {\rm and} \qquad \kappa_g:=\kappa(\cdot,g).$$
Recall that the tangent bundle $TG$ of $G$ is a Lie group itself with the  product and  the inverse defined as follows
$$\label{tangentG}
X_g \bullet Y_h:= TL_g(h)Y_h+TR_h(g)X_g,\qquad X_g^{-1}:=-TL_{g^{-1}}(e)\circ TR_{g^{-1}}(g)X_g,
$$
where $X_g\in T_g G$, $Y_h\in T_h G$ and $L_g(h):=gh$, $R_g(h):=hg$. Let $e\in G$ be the unit element of $G$ and $\mathbf{0}:G\to TG$ be the zero section of the tangent bundle $TG$. Then one has 
$$\label{tangentG1}
X_e\bullet Y_e=X_e+Y_e, \qquad \mathbf{0}_g\bullet \mathbf{0}_h= \mathbf{0}_{gh},$$
$$ X_g\bullet Y_e\bullet X_g^{-1}=(TR_{g^{-1}}(g)\circ TL_{g}(e))Y_e=:Ad_g Y_e
$$
So, the corresponding Lie algebra $T_eG$ can be considered as an abelian normal subgroup of $TG$ and the zero section $\mathbf{0}:G\to TG$ is a group monomorphism.

The diffeomorphism 
\be\label{IG} I:G\times T_eG\ni(g,X_e)\mapsto TR_g(e)X_e=:X_g\in TG\ee
allows us to consider $TG$ as the semidirect product $G\ltimes_{Ad_G}T_eG$ of $G$ by the $T_eG$, where the group product  of $(g,X_e), (h,Y_e)\in G\ltimes_{Ad_G}T_eG$ is given by
$$\label{TGprod}(g,X_e)\bullet (h,Y_e)=I^{-1}(I(g,X_e)\cdot I(h,Y_e) )=$$
$$=(gh,X_e+T(R_{g^{-1}}\circ L_g)(e)Y_e)=(gh,X_e+Ad_gY_e).$$

Using the Lie group isomorphism (\ref{IG}) and the equality 
$$\label{w1}\kappa_g\circ\kappa_p=\kappa_p\circ R_g,$$
we obtain the action 
\be\label{actionTG2} \Phi_{(g,X_e)}(v_p)=T\kappa_g(p)(v_p+T\kappa_p(e)X_e)\ee
of $G\ltimes_{Ad_G}T_eG$ on the tangent bundle $TP$.

Applying  the above action one obtains the following  isomorphisms 
\begin{gather}
\label{v1} TP/T^vP\cong TP/T_eG ,\\
\label{v2} TP/TG\cong (TP/T_eG)/G\cong  (TP/G)/T_eG,\\
\label{v3} TM=T(P/G)\cong TP/TG,
\end{gather}
of vector bundles,
where we write $T^vP := Ker T\mu$ for the vertical subbundle of  $TP$. These isomorphisms will be useful in subsequent considerations.

\bigskip

Another group important here which acts on $TP$ is the group $Aut_0TP$ of smooth automorphisms 
  $A:TP\to TP$ of the tangent bundle covering the identity map of $P$, i.e.  for any $p\in P$ one has  the map  $A(p):T_pP\to T_pP$ which  is an isomorphism of the tangent space  $T_pP$ and $A(p)$ depends smoothly  on $p$. Note here that $Aut_0TP$  is a normal subgroup of the group $AutTP$ of all automorphisms of $TP$. 
	
	By $Aut_{TG}TP\subset Aut_0TP$ we denote the subgroup consisting of those elements of $Aut_0TP$ whose action on $TP$ commutes with the action (\ref{actionTG2}) of $TG\cong G\ltimes _{Ad_g}T_eG$ on $TP$, i.e.
\be \label{equiv} A(pg)\circ \Phi_{(g,X_e)}=\Phi_{(g,X_e)}\circ A(p).\ee
From the isomorphisms (\ref{v2}) and (\ref{v3}) it follows that the group $Aut_{TG}TP$ acts also on vector bundles $TP/G\to M$ and  $TM\to M$.

\begin{prop}\label{prop:1}

 $A\in  Aut_{TG}TP$ if and only if 
\be\label{Ap1} A(p)\circ T\kappa_p(e)=T\kappa_p(e)\ee
\be\label{Ap2} A(pg)\circ T\kappa_g(p)=T\kappa_g(p)\circ A(p)\ee
for any $g\in G$ and $p\in P$.
\end{prop}

 \begin{proof} Substituting (\ref{actionTG2}) into (\ref{equiv}) we obtain the equality
\be\label{Apg} A(pg)[T\kappa_g(p)v_p+(T\kappa_g(p)\circ T\kappa_p(e))X_e] = T\kappa_g(p)A(p)v_p+(T\kappa_g(p)\circ T\kappa_p(e))X_e,\ee
which is valid for all $v_p\in T_pP$ and $X_e\in T_eG$.  From the relation 
$$ \kappa_g\circ \kappa_p=\kappa_{pg}\circ I_g^{-1},$$
where $I_g:=L_g\circ R_g^{-1}$ and $Ad_g=TI_g$, setting $v_p=0$ in (\ref{Apg})   we obtain 
 the equality
$$ (A(pg)\circ T\kappa_{pg}(e)\circ Ad_{g^{-1}}(e))X_e=(T\kappa_{pg}(e)\circ Ad_{g^{-1}}(e))X_e,$$
valid for any $X_e\in T_eG$. The above gives (\ref{Ap1}). In order to show (\ref{Ap2}) we substitute $X_e=0$ into (\ref{Apg}) and apply the same arguments as for the previous case.
\end{proof}

Now we define the subgroup $Aut_N TP\subset Aut_{TG}TP$ consisting of $A\in Aut_{TG}TP$ such that $A(p)=\id_p+B(p)$, where $B(p):T_pP\to T^v_pP$. Conditions (\ref{Ap1}) and (\ref{Ap2}) imposed on $A(p)$ written in terms of $B(p)$ take the form 

$$\begin{array}{l}\label{B1} B(p)\circ T\kappa_p(e)=0\\
 B(pg)\circ T\kappa_g(p)=T\kappa_g(p)\circ B(p).\end{array}$$
From the definition  of $B(p)$ and (\ref{B1}) one has 
\be\label{Im}ImB(p)\subset T^v_pP\subset Ker B(p).\ee Thus it follows that $B_1(p)B_2(p)=0$ for any $\id+B_1,\ \id+B_2\in Aut _NP$. So, one has 
$$ A_1(p)\circ A_2(p)=(\id_p+B_1(p))(\id_p+B_2(p))=\id_p+B_1(p)+B_2(p),$$
for $A_1(p),A_2(p)\in Aut_N TP$. This shows that $Aut_N TP$ is a commutative subgroup of $Aut_{TG}TP$. Therefore, we may identify  $Aut_N TP$ with the vector subspace $End_N TP$  of $EndTP$ which consists of such endomorphisms  $B(p):T_pP\to T_pP$ that the property (\ref{Im}) is valid for any $p\in P$.

\bigskip

Now, let us recall that by the definition a connection form on $P$ is a $T_eG$-valued differential one-form $\alpha$ satisfying the conditions
\be\label{alpha1} \alpha_p\circ T\kappa_p(e)=\id_{T_eG},\ee
\be\label{alpha2} \alpha_{pg}\circ T\kappa_g(p)=\Ad_{g^{-1}}\circ \alpha_p\ee
for the value $\alpha_p$ of $\alpha$ at $p\in P$ and $g\in G$. Using $\alpha$ one defines the decomposition 
\be\label{decompTp} T_pP=T^v_pP\oplus T_p^{\alpha,h}P\ee
of $T_pP$ on the vertical $T^v_pP$ and the horizontal $T^{\alpha,h}_pP:=Ker \alpha_p$ subspaces. 
 Using  the decomposition (\ref{decompTp})  one defines the vector spaces isomorphism 
\be\label{isom}\Gamma_\alpha(p):T_{\mu(p)}M\stackrel{\sim}{\to} T^{\alpha,h}_pP\ee
 such that 
$$\label{Gamma} \Gamma_\alpha(pg)=T\kappa_g(p)\circ\Gamma_\alpha(p)\qquad{\rm and}\qquad \id _{T_{\mu(p)}M}=T\mu(p)\circ \Gamma_\alpha(p).$$
Let us take the decomposition
\be\label{iddecomp} \id_{T_pP}=\Pi^v_\alpha(p)+\Pi^h_\alpha(p)\ee
of the identity map of $T_pP$ into the sum of projections corresponding to (\ref{decompTp}). Then we have 
\be \label{projv}\Pi_\alpha^h(p)=\Gamma_\alpha(p)\circ T\mu(p)\quad{\rm and}\quad \Pi_p^v(p)=T\kappa_p(e)\circ \alpha_p.\ee

\begin{prop}\label{prop:2}\ben
\item[(i)]
One has the  short exact sequence 
\be\label{ses} \{\id_{TP}\}\to Aut_N TP\stackrel{\iota}{\to} Aut_{TG}TP\stackrel{\lambda}{\to} Aut_0TM\to \{\id_{TM}\}\ee
of the group morphisms, where $\iota$ is the inclusion map and ${\lambda}$ is an epimorphism of $Aut_{TG}TP$ on the group $Aut_0TM$ of the automorphisms of the tangent space $TM$ covering the identity map of $M$ defined by
\be\label{piA} ({\lambda}(A)(\mu(p))(T\mu(p))v_p):=(T\mu(p)\circ A(p))v_p,\ee
where $v_p\in T_pP$.
\item[(ii)] Fixing a connection $\alpha$ one defines the injection $\sigma_\alpha:Aut_0TM\to Aut_{TG}TP$ by 
\be \label {tildaA}\sigma_\alpha(\tilde A)(p):=\Pi^v_\alpha(p)+\Gamma_\alpha(p)\circ \tilde A(\mu(p))\circ T\mu(p),\ee
where $\tilde A\in Aut_0 TM$, and the surjection $\beta_\alpha:Aut_{TG}TP\to Aut_NTP$ by
\be\label{betaalpha} \beta_\alpha(A):=A\sigma_\alpha(\lambda(A))^{-1},\ee
where $A\in Aut_{TG}TP$,
which are arranged into the short exact sequence 
\be\label{sesinv} \{\id_{TM}\}\to Aut_0 TM\stackrel{\sigma_\alpha}{\longrightarrow} Aut_{TG}TP\stackrel{\beta_\alpha}{\longrightarrow} Aut_NTP\to\{\id_{TP}\},\ee
inverse to the sequence (\ref{ses}), i.e. $ Im\sigma_\alpha=\beta_\alpha^{-1}(\id_{TP})$, $\sigma_\alpha$ is a right inverse $\pi\circ\sigma_\alpha=\id_{TM}$ of $\pi$ and $\beta_\alpha$ is the left inverse $\beta_\alpha\circ\iota=\id_{TP}$ of $\iota$. The map $\sigma_\alpha$  is a monomorphism
$$ \sigma_\alpha(\tilde A_1\tilde A_2)=\sigma_\alpha(\tilde A_1)\sigma_\alpha(\tilde A_2)$$
of the groups and $\beta_\alpha$ satisfies 
$$ \beta_\alpha(A_1 A_2)=\beta_\alpha(A_1)\sigma_\alpha(\lambda(A_1))\beta_\alpha(A_2)\sigma_\alpha(\lambda(A_1))^{-1}.$$

\item[(iii)] The decomposition 
\be\label{Apdecomp} A(p)=(\id_p+B(p))\sigma_\alpha(\tilde A)(p)\ee
of $A\in Aut_{TG}TP$, where $\id_p+B(p)\in Aut_NTP$ and $\tilde A\in Aut_0 TM$, defines an isomorphism of $Aut_{TG}TP$ with the semidirect product group $Aut_0TM\ltimes_\alpha End_N TP$, where the product of $(\tilde A_1,B_1), (\tilde A_2,B_2)\in Aut_0TM\ltimes_\alpha End_N TP$ is given by 
\be\label{prod} [(\tilde A_1,B_1)\cdot (\tilde A_2,B_2)](p):=(\tilde A_1(\mu(p))\tilde A_2(\mu(p)), B_1(p)+B_2(p)\circ \Gamma_\alpha(p)\circ \tilde A_1^{-1}(\mu(p))\circ T\mu(p)).\ee
\een
\end{prop}
\begin{proof}  From the definition (\ref{piA}) of $\pi$, for all $A_1,A_2 \in Aut_{TG}TP$, we obtain that 
$$ \lambda(A_1A_2)(\mu(p))(T\mu(p)v_p)=(T\mu(p)\circ A_1(p))(A_2(p)v_p)=$$
$$=\lambda(A_1)(\mu(p))(T\mu(p)\circ A_2(p))v_p=(\lambda(A_1)(\mu(p))\circ\lambda(A_2(\mu(p)))(T\mu(p)v_p)$$
for $T\mu(p)v_p\in T_{\mu(p)}M$. Also, for $\tilde A_1,\tilde A_2\in Aut _0 TM$, from (\ref{tildaA}) we have 
$$ \sigma_\alpha(\tilde A_1)(p)\circ \sigma_\alpha(\tilde A_2)(p)=$$
$$=(\Pi^v_\alpha(p)+\Gamma_\alpha(p)\circ \tilde A_1(\mu(p))\circ T\mu(p))(\Pi^v_\alpha(p)+\Gamma_\alpha(p)\circ \tilde A_2(\mu(p))\circ T\mu(p))=$$
$$=\Pi^v_\alpha(p)+\Gamma_\alpha(p)\circ \tilde A_1(\mu(p))\circ T\mu(p)\circ\Gamma_\alpha(p)\circ \tilde A_2(\mu(p))\circ T\mu(p)=$$
$$=\Pi^v_\alpha(p)+\Gamma_\alpha(p)\circ \tilde A_1(\mu(p))\circ  \tilde A_2(\mu(p))\circ T\mu(p)=\sigma_\alpha(\tilde A_1\tilde A_2)(p).$$
In order to obtain these equalities we used ${\Pi_\alpha^v}^2=\Pi_\alpha^v$ and $\Pi_\alpha^v\circ \Gamma_\alpha(p)=0$ and $T\mu(p)\circ \Pi_\alpha^v=0$. Additionally we  have the following equalities
$$(\beta_\alpha\circ\iota)(\id_{TP}+B)=(\id_{TP}+B)\sigma_\alpha(\id_{TP})=\id_{TP}+B,$$
$$(\lambda\circ\beta_\alpha)(A)=\lambda(A\sigma_\alpha(\lambda(A))^{-1})=\lambda(A)\lambda(A)^{-1}=\id_{TM}$$
and
$$ (\lambda\circ \sigma_\alpha)(\tilde A)(\mu(p))(T\mu(p)v_p)=(\lambda(\sigma_\alpha(\tilde A))(\mu(p))(T\mu(p)v_p)=T\mu(p)(\sigma_\alpha(\tilde A)(p)v_p)=$$
$$=T\mu(p)(\Pi^v_\alpha(p)+\Gamma_\alpha(p)\circ \tilde A(\mu(p))\circ T\mu(p))v_p=\tilde A(\mu(p))(T\mu(p)v_p).$$

\bigskip

Let  us also note that $\beta_\alpha(A)=\id_{TP}$  if and only if $A=\sigma_\alpha(\lambda(A))$. This implies that $ \beta_\alpha^{-1}(\id_{TP})=Im \sigma_\alpha$. Summing up the above statements we prove the points (i) and (ii) of the proposition.

In order to prove (iii) we  first note  that  the decomposition (\ref{Apdecomp}) of $A\in Aut_{TG}TP$ into the product of $\id_{TP}+B\in Aut_N TP$ and $\sigma_\alpha(\tilde A)\in\sigma_\alpha(Aut_0 TM)$ follows from $\lambda\circ\sigma_\alpha=\id_{TM}$ and from $Aut_NTP=ker\lambda$.  Taking this fact into account one finds 
$$ \sigma_\alpha(\tilde A)(p)(\id_p+B(p)) \sigma_\alpha(\tilde A^{-1})(p)=\id_p+B(p)\circ \Gamma_\alpha(p)\circ \tilde A^{-1}(\mu(p))\circ T\mu(p)$$
which yields (\ref{prod}).

\end{proof}

The structural properties of $Aut_{TG}TP$ described in the above proposition will be useful for the subsequent considerations.

\section{ The action of $ Aut_{TG}TP$ on the space of connections}\label{sec3}

In this section we describe the relationship between 
 the space $Conn P(M,G)$ of all connections on $P(M,G)$ and the groups from the short exact sequence (\ref{ses}). To this end we define  by\be\label{phiA} \phi_A(\alpha)_p:=\alpha_p\circ A(p)^{-1}\ee
the left action $\phi_A:Conn P(M,G)\to Conn P(M,G)$ of $Aut_{TG}TP$ on $ConnP(M,G)$, i.e. $\phi$ satisfies $\phi_{A_1A_2}=\phi_{A_1}\circ\phi_{A_1}$ for $A_1,A_2\in Aut_{TG}TP$.


\begin{prop}\label{prop:3} For the groups $Aut_{TG}TP$, $Aut_N TP$ and $Aut_0TM$ one has:

(i) The action  of $Aut_{TG}TP$ defined in (\ref{phiA}) is transitive.

(ii) The horizontal lift  $\Gamma_\alpha$ defined by $\alpha\in ConnP(M,G)$, see (\ref{isom}), satisfies the relation
\be\label{equivA} A(p)\circ \Gamma_\alpha(p)=\Gamma_{\phi_{A}(\alpha)}(p)\circ \lambda(A)(\mu(p)) \ee
for all $A\in Aut_{TG}TP$.

(iii) The action (\ref{phiA}) restricted to the subgroup $Aut_N TP$ is  free and transitive.

(iv) The subgroup $\sigma_\alpha(Aut_0TM)$ is the stabilizer of $\alpha$ with respect to the action (\ref{phiA}).

\end{prop}
\begin{proof}

(i) In order to show that $\phi_A(\alpha)\in  Conn P(M,G)$ we note that it satisfies conditions (\ref{alpha1}) and (\ref{alpha2}) if $A$ satisfies  (\ref{Ap1}) and (\ref{Ap2}). Next let us take the decompositions 
$$\begin{array}{l}
v_p=v_p^v+v_p^h\\
v_p=v_p^{v'}+v_p^{h'}\end{array}$$
of $v_p\in T_pP$ on the vertical and horizontal parts with respect to the connections $\alpha$ and $\alpha '$, respectively. One easily sees that $A(p):T_pP\to T_pP$, defined by 
\be \label{Ap} A(p)v_p:=v_p^v+v_p^{h'},\ee
satisfies  conditions (\ref{Ap1}) and (\ref{Ap2}) and $\alpha'=\phi_A(\alpha)$. Thus 
$A\in  Aut_{TG}(TP)$ and  the action (\ref{phiA}) is transitive.

(ii) The equivariance property (\ref{equivA}) follows from $Ker\phi_A(\alpha_p)=A(p)Ker\alpha_p$ and from 
$T^{\alpha,h}_pP=Ker\alpha_p.$

(iii) For any two connections $\alpha$, $\alpha '$  their difference $\alpha-\alpha '$ is a $T_eG$-valued  tensorial one-form, i.e. $Ker(\alpha-\alpha ')=T^v_pP$ and $(\alpha_{pg}-\alpha '_{pg})\circ T\kappa_g(p)=Ad_{g^{-1}}(\alpha_p-\alpha '_p)$ for any $p\in P$. Since, for any $p\in P$,  $\alpha$ define  the vector space isomorphism $\alpha_p:T^v_pP\to T_eG$ which is the inverse to $T\kappa_p(e):T_eG\to T^v_pP$, see (\ref{alpha1}), it follows that 
\be\label{Bp} B(p):=T\kappa_p(e)\circ (\alpha_p-\alpha_p ')\ee
satisfies (\ref{Im}). From (\ref{Bp}) one obtains $\alpha'=\phi_A(\alpha)$ where $A(p):=\id_p+B(p)$. The above proves point (iii).

(iv) Straightforward verification.
\end{proof}

The next proposition shows that one can define the subgroup $Aut_{TG}TP\subset Aut_0 TP$ in terms of the connection space  $Conn P(M,G)$.
\begin{prop}\label{prop:4} If $A\in Aut_0(TP)$ and $\phi_A( Conn P(M,G))\subset Conn P(M,G)$ then $A\in  Aut_{TG}(TP)$.\end{prop}

\begin{proof} Let $A\in Aut(TP)$ be such that $\phi_A(Conn P(M,G))\subset Conn P(M,G)$ then for any $\alpha\in ConnP(M,G)$ one has 
$$ \alpha_p\circ A(p)\circ T\kappa_p(e)=\alpha_p\circ T\kappa_p(e)$$
and
$$ \alpha_{pg}\circ A(pg)^{-1}\circ T\kappa_g(e)=\alpha_{pg}\circ T\kappa_g(p)\circ A(p)^{-1}.$$
The above equalities imply   (\ref{Ap1}) and (\ref{Ap2}) if 
$$\label{cap} \bigcap_{\alpha\in Conn P(M,G)} Ker \ \alpha_p=\{0\}$$
for any $p\in P$.
In order to prove (\ref{cap}) we observe that for any vector subspace $H_p\subset T_pP$ transversal to $T^v_pP$ and an open subset 
 $\Omega\subset M$ such that $\mu^{-1}(\Omega)\cong G\times \Omega_p$ and $\mu(p)\in \Omega$ there exists a local connection form $\alpha$ on $\mu^{-1} (\Omega)$ for which $Ker\ \alpha_p=H_p$. Assuming paracompactness of $M$ and using the decomposition of the unity for the properly chosen covering of $M$ by $\Omega_p$ one can construct a connection form $\alpha$ on $M$ such that  $H_p=Ker\ \alpha_p$. \end{proof}


Let us note that, given an arbitrary  $c\in \mathbb{R}\setminus \{0\}$, any connection $\alpha\in Conn P(M,G)$ defines 
 a multiplicative one-parameter subgroup of $ Aut_{TG}TP$, i.e.
$A_\alpha^{c_1}\circ A_\alpha^{c_2}=A_\alpha^{c_1c_2}$, for $c_1,c_2\in \mathbb{R}\setminus \{0\}$ by 
$$\label{Aalpha} A_\alpha^c(p):=\Pi_\alpha^v(p)+c\Pi_\alpha^h(p)=\sigma_\alpha(c \ \id_{TM}).$$


\section{ Space of generalized canonical forms on $T^*P$}\label{sec4}

We recall for further considerations that the standard symplectic form on $T^*P$ is $\omega_0=d\gamma_0$, where $\gamma_0\in C^\infty T^*(T^*P)$ is the canonical one-form on $T^*P$ defined at $\varphi\in T^*P$ by 
$$\label{gamma0} \langle\gamma_{0\varphi},\xi_\varphi\rangle:=\langle\varphi, T\pi^*(\varphi)\xi_\varphi\rangle,$$
where $\pi^*:T^*P\to P$ is the projection of $T^*P$ on the base and $\xi_\varphi\in T_\varphi(T^*P)$.

Let us mention also that   by definition a \textit{linear vector field} on $T^*P$ is a pair $(\xi,\chi)$ of vector fields $\xi\in C^\infty T(T^*P)$ and $\chi\in C^\infty TP$ such that 
 \unitlength=5mm $$\label{bundle}\begin{picture}(11,4.6)
    \put(1,4){\makebox(0,0){$T^*P$}}
    \put(8,4){\makebox(0,0){$T(T^*P)$}}
    \put(1,0){\makebox(0,0){$P$}}
    \put(8,0){\makebox(0,0){$TP$}}
     \put(1,3){\vector(0,-1){2}}
    \put(8,3){\vector(0,-1){2}}
       \put(2.5,4){\vector(1,0){3.7}}
    \put(2.7,0){\vector(1,0){3.7}}
    \put(0.4,2){\makebox(0,0){$\pi^*$}}
      \put(9,2){\makebox(0,0){$T\pi^*$}}
     \put(4.5,4.3){\makebox(0,0){$\xi$}}
    \put(4.5,0.5){\makebox(0,0){$ \chi $}}
    \end{picture}$$
    \newline
		defines  a morphism of vector bundles. Note here that $T\pi^*(\varphi)\xi_\varphi=\chi_{\pi^*(\varphi)}$. Regarding the theory of linear vector fields over vector bundles see e.g. Section 3.4 of \cite{Mackenzie:GT}, where their various properties are discussed.

		In the sequel we will denote by $LinC^\infty T(T^*P)$ the Lie algebra of linear vector fields over the vector bundle $\pi^*:T^*P\to P$. The Lie bracket of $(\xi_1,\chi_1)$, $(\xi_2,\chi_2)\in LinC^\infty T(T^*P)$ is defined by 
		$$[(\xi_1,\chi_1),(\xi_2,\chi_2)]:=([\xi_1,\xi_2],[\chi_1,\chi_2])$$
		and the vector space structure  on $LinC^\infty T(T^*P)$ by 
		$$c_1(\xi_1,\chi_1)+c_2(\xi_2,\chi_2):=(c_1\xi_1+c_2\xi_2,c_1\chi_1+c_2\chi_2).$$
		
		Let  $LinC^\infty( T^*P)$ denote the vector space of smooth fibre-wise linear functions on $T^*P$. Notice that spaces $LinC^\infty (T(T^*P))$ and $Lin C^\infty(T^*P)$ have structures of $C^\infty(P)$-modules defined by 
		$f(\xi,\chi):=((f\circ\pi^*)\xi,f\chi)$ and by $fl:=(f\circ\pi^*)l$, respectively, where $f\in C^\infty(P)$ and $l\in Lin C^\infty(T^*P)$.
		
		\begin{defn}\label{def1} A differential one-form $\gamma\in C^\infty T^*(T^*P)$ is called a \textsl{generalized canonical form} on $T^*P$ if:
		\ben
		\item[(i)] $\gamma_\varphi\not=0$ for any $\varphi\in T^*P$,
		\item[(ii)]  $ker T\pi^*(\varphi)\subset ker\ \gamma_\varphi:=\{\xi_\varphi\in T_\varphi(T^*P):\ \langle \gamma_\varphi,\xi_\varphi\rangle=0\} $,
		\item[(iii)]  $\langle \gamma,\xi\rangle\in Lin C^\infty(T^*P)$ for any $\xi\in Lin C^\infty T(T^*P)$.
		\een
				\end{defn}

		
The space of generalized canonical forms on $T^*P$  will be denoted by $Can T^*P$. Let us note here that $\gamma_0\in Can T^*P$.


		\begin{prop}\label{prop:5}
		\ben
		\item[(i)]
		The map 
	$\label{AAut0} \Theta:Aut_0 TP\to Can T^*P$
		defined by 
	\be\label{gammaA} \langle\Theta(A)_\varphi,\xi_\varphi\rangle:=\langle\varphi,A(\pi^*(\varphi)) T\pi^*(\varphi))\xi_\varphi\rangle,\ee
	where $\xi_\varphi\in T_\varphi(T^*P)$, is  bijective.
		\item[(ii)]  The natural  left action 
		$\label{Gamma*} L^*:Aut_0TP\times CanT^*P\to Can T^*P$
		of $Aut _0 TP$ on $Can T^*P$ defined by 
		\be\label{A*} \langle (L^*_A(\gamma))_\varphi,\xi_\varphi\rangle:=\langle \gamma_{A^*(\varphi)},TA^*(\varphi)\xi_\varphi\rangle,\ee
		 where $A^*:T^*P\to T^*P$ is the dual of $A\in Aut_0 TP$, is a transitive and free action.
		Furthermore, 
		\be\label{phi*} {L}^*_A\circ\Theta=\Theta\circ L_A,\ee
		where $L_A {A}':=A{A}'$, i.e. 
	 ${L}^*_A\Theta(A')=\Theta(AA').$
				\een
		\end{prop}
		\begin{proof} (i) If $\Theta(A_1)=\Theta(A_2)$, then using (\ref{A*}), for any $T\pi^*(\varphi)\xi_\varphi\in T_{\pi^*(\varphi)}P$ we obtain 
		$$\langle (A_1(\pi^*(\varphi))^*-A_2(\pi^*(\varphi))^*)\varphi, T\pi^*(\varphi)\xi_\varphi\rangle=0.$$
		This gives $A_1=A_2$. So, $\Theta$ is an injection.
		
		Let us take $\gamma\in CanT^*P$. Then $\langle \gamma,\xi\rangle \in Lin C^\infty (T^*P)$ if $\langle \xi,\chi\rangle\in LinC^\infty (T(T^*P))$. By virtue of the point (ii) of Definition \ref{def1} the fibre-wise linear functions $\langle \gamma,\xi\rangle$ depend only on  vector fields $\chi=T\pi^*\xi\in C^\infty TP$ and this dependence defines a morphism of $C^\infty(P)$-modules. On the other hand one can consider $\langle \gamma,\xi\rangle$ as a section of $T^{**}P\cong TP$. Thus we can represent it as
		\be\label{chi'} \langle \gamma,\xi\rangle(\varphi)=\langle\varphi,\chi'(\pi^*(\varphi)\rangle\ee
		by some vector field $\chi'\in C^\infty TP$.  The dependence between $\langle \gamma,\xi\rangle$ and $\chi'$ given by (\ref{chi'}) is also a morphism of $C^\infty(P)$-modules. Therefore, there exists $A\in End_0TP$ such that 		
				\be\label{chi'2} \chi '(p)={A}(p)\chi(p)\ee 
				and we have $\gamma=\Theta(A)$. 
			Substituting (\ref{chi'2} ) into (\ref{chi'} ) we obtain 
				$$\label{d1}  \langle \gamma_\varphi,\xi_\varphi\rangle=\langle \varphi,A(\pi^*(\varphi))T\pi^*(\varphi)\xi_\varphi\rangle=
				\langle A(\pi^*(\varphi))^*\varphi,T\pi^*(\varphi)\xi_\varphi\rangle$$
				and, thus $\gamma=\Theta(A)$.
				
										Let us assume that $A\notin Aut_0TP$. Then there exists $\varphi$ such that $A(\pi^*(\varphi))^*\varphi=0$. From (\ref{d1}) we see that for this $\varphi$ we have $\gamma_\varphi=0$, which contradicts the point (i) of Definition \ref{def1}. So, $A\in Aut_0TP$ and thus $\Theta$ is a surjection. The above proves (i).


(ii) Since any element of $Can T^*P$ can be written as $\Theta(A')$ for some ${A}'\in Aut_0 TP$ we obtain from the definition (\ref{A*}) that 
$$ \langle L^*_A(\Theta(A'))_\varphi,\xi_\varphi\rangle=\langle\Theta(A')_{A^*(\varphi)},TA^*(\varphi)\xi_\varphi\rangle=
\langle{A^*(\varphi)},{A}'(\pi^*(\varphi))\circ T\pi^*(\varphi)\circ TA^*(\varphi)\xi_\varphi\rangle=$$
$$=\langle{\varphi}, A(\pi^*(\varphi))\circ{A}'(\pi^*(\varphi))\circ T(\pi^*\circ A^*)(\varphi)\xi_\varphi\rangle=
\langle{\varphi}, A(\pi^*(\varphi))\circ{A}'(\pi^*(\varphi))\circ T(\pi^*)(\varphi)\xi_\varphi\rangle=$$
$$=\langle \Theta(AA')_\varphi,\xi_\varphi\rangle,$$
which  proves 
 (\ref{phi*}). From (\ref{phi*}) and from the point (i) of the proposition it follows that  $L^*$ is a transitive and free action. 
		\end{proof}
		
		From the above proposition we conclude that $\gamma\in CanT^*P$ is the pull-back $\gamma=\Theta(A)=L^*_A\gamma_0$ of the canonical form $\gamma_0$. So, $\omega_A:=d\Theta(A)$ is a symplectic form.
		
		\bigskip

		  The lift  $\Phi^*_g:T^*P\to T^*P$ of the action $\kappa_g:P\to P$ to the cotangent bundle $T^*P$ is defined by 
		\be\label{62} \Phi^*_g(\varphi)(pg)=(T\kappa_g(p)^{-1})^*\varphi\ee
		where $p=\pi^*(\varphi)$.
		
		If $\gamma\in Can T^*P$ is  $G$-invariant  with respect to (\ref{62}), then $\mathcal{L}_{\xi^X}\gamma=0$ for  $X\in T_eG$, where $\xi^X\in C^\infty T(T^*P)$ is the fundamental  vector field, i.e. the vector field tangent to the flow $t\to \Phi^*_{\exp tx}$. 	
		So, for a $G$-invariant symplectic  form $\omega_A=d\gamma=d\Theta(A)$ one has   
		$$ \xi^X\llcorner \omega=-d\langle J_A,X\rangle$$
			where the $G$-equivariant
							momentum map $J_A:T^*P\to T_e^*G$ is given by $ J_A=J_0\circ A^*.$
		We note here that 
		for the standard symplectic form  $\omega_0=d\gamma_0$  the momentum map  is
		$$\label{mJ0} J_0(\varphi)=\varphi\circ T\kappa_{\pi^*(\varphi)}(e)$$
		
		It is reasonable to define the space 
		$$\label{izoCan} Can_{TG} T^*P:=\Theta(Aut_{TG}TP)$$
		which is an $Aut_{TG}TP$-invariant subspace  of the space $Can T^*P$.

\begin{prop}\label{prop6} 
\ben
\item[(i)] The generalized canonical form $\Theta(A)$  belongs to $Can_{TG}T^*P$ if and only if  $(\phi_g^*)^*\Theta(A)=\Theta(A)$ and $J_A=J_0$.
\item [(ii)] One can consider $Can_{TG} T^*P$ as the orbit of the subgroup $Aut_{TG}TP\subset Aut_0 TP$ taken through $\gamma_0$ with respect to the free action  $L^*$ defined in  (\ref{A*}).
\item[(iii)] If $A\in Aut_0TP$ and ${L}^*_A(Can_{TG} T^*P)\subset Can_{TG} T^*P$ then $A\in Aut_{TG}TP$.
\een
\end{prop}
\begin{proof}
 (i) The canonical form $\Theta(A)$ is $G$-invariant if and only if 
\be\label{d1}\langle \Theta(A)_{\Phi^*_g(\varphi)},T\Phi^*_g(\varphi)\xi_\varphi\rangle=\langle\Theta(A)_\varphi,\xi_\varphi\rangle\ee
for any $g\in G$.
For the left hand side of (\ref{d1}) we have 
\be\label{proof1} \langle (\phi_g^*)^*\Theta(A)_\varphi,\xi_\varphi\rangle=\langle \Theta(A)_{\Phi^*_g(\varphi)},T\Phi^*_g(\varphi)\xi_\varphi\rangle=\ee
$$=\langle{\Phi^*_g(\varphi)},A(\pi^*(\Phi^*_g(\varphi))T\pi^*(\Phi^*_g(\varphi))T\Phi^*_g(\varphi)\xi_\varphi\rangle=$$
$$=\langle\varphi,T\kappa_g(p)^{-1}A(\pi^*\circ\Phi_g^*)(\varphi)T(\pi^*\circ\Phi_g^*)(\varphi)\xi_\varphi\rangle=
$$
$$=\langle\varphi,T\kappa_g(p)^{-1}A(\pi^*(\varphi)g)T\kappa_g(\pi^*(\varphi))T\pi^*(\varphi)\xi_\varphi\rangle=\langle\Theta(A)_\varphi,\xi_\varphi\rangle$$
for any $\varphi\in T^*P$ and $\xi_\varphi\in T_\varphi(T^*P)$. Note that the second equality in (\ref{proof1}) follows from $\pi^*\circ \Phi^*_g=\kappa_g\circ \pi^*$. From (\ref{proof1}) and from the definition (\ref{gammaA}) we obtain  the condition (\ref{Ap2}).

From  $J_A=J_0$ we have 
$$ \langle J(\varphi),X\rangle=\langle \Theta(A)_\varphi,\xi^X_\varphi\rangle=\langle \varphi,A(\pi^*(\varphi))T\kappa_{\pi^*(\varphi)}(e)X\rangle=$$
$$=\langle \varphi,T\kappa_{\pi^*(\varphi)}(e)X\rangle$$
for all $\varphi\in T^*P$ and $X\in T_eG$. This shows that an element $A\in Aut_0TP$ satisfies (\ref{Ap1}).

(ii) This statement follows from (i) and from ${L}^*_A\gamma_0=\Theta(A)$.

(iii) If ${L}^*_ACan_{TG} T^*P\subset Can_{TG} T^*P$ then ${L}^*_A\Theta(A_1)=\Theta(AA_1)\in Can_{TG} T^*P$ for any $A_1\in Aut_{TG}TP$. So, due to the property (i) we have $AA_1\in  Aut_{TG}TP$ and,  thus $A\in Aut_{TG}TP.$

\end{proof}


In Section \ref{sec2}  we  fixed a reference connection $\alpha$ in order to investigate the structure of group $Aut_{TG}TP$, see Proposition \ref{prop:2}. Now, taking into consideration  Proposition \ref{prop6} we study the structure of generalized canonical forms $\Theta(A)\in Can_{TG}T^*P$ using decompositions (\ref{Apdecomp})
 and (\ref{iddecomp}), and (\ref{projv}). We obtain 
\be\label{Ap} A(p)=\Pi^v_\alpha(p)+\Pi^v_\alpha(p)A(p)\Pi^h_\alpha(p)+\Pi^h_\alpha(p)A(p)\Pi^h_\alpha(p)=\ee
$$=T\kappa_p(e)\circ \alpha_p+(\id _{TP}+B)(p)\circ \Gamma_\alpha(p)\circ \tilde A(\mu(p))\circ T\mu(p),$$
where $\id _{TP}+B\in Aut_NTP$ and $\tilde A\in Aut_0TM$. Substituting $A$ given by (\ref{Ap}) into the definition (\ref{gammaA}) we find the corresponding formula for $\Theta(A)$
\be\label{ThetaA} \Theta(A)(\varphi)=\varphi\circ T\kappa_{\pi^*(\varphi)}\circ \alpha_{\pi^*(\varphi)}+\ee
$$+\varphi\circ (\id _{TP}+B)(\pi^*(\varphi))\circ \Gamma_\alpha(\pi^*(\varphi))\circ\tilde A((\mu\circ \pi^*)(\varphi))\circ T(\mu\circ \pi^*)(\varphi).$$
In particular cases when $\tilde A=\id_{TM}$ and $B=0$ we have
\be\label{Theta2} \Theta(\id +B)(\varphi)=\varphi\circ T\pi^*(\varphi)+\varphi\circ B(\pi^*(\varphi))\circ T\pi^*(\varphi)=\ee
$$=\varphi\circ T\pi^*(\varphi)+\varphi\circ T\kappa_{\pi^*(\varphi)}(e)\circ(\alpha'_{\pi^*(\varphi)}-\alpha_{\pi^*(\varphi)})\circ T\pi^*(\varphi)$$
and
$$ \Theta(\sigma_\alpha(\tilde A))(\varphi)=J_0(\varphi)\circ \alpha_{\pi^*(\varphi)}\circ T\pi^*(\varphi)+\varphi\circ \Gamma_\alpha(\pi^*(\varphi))\circ\tilde A((\mu\circ\pi^*)(\varphi))\circ T(\mu\circ \pi^*)(\varphi),$$
respectively.
Let us note  that $\Theta(\sigma_\alpha(\id_{TM}))=\Theta(\id_{TP})=\varphi\circ T\pi^*(\varphi)$, i.e. $\Theta(\sigma_\alpha(\id_{TM}))$ is the canonical one-form $\gamma_0$.

\begin{cor}\label{cor:7} Fixing a connection $\alpha$ one obtains from (\ref{Theta2}) an embedding $\iota_\alpha:ConnP(M,G)\hookrightarrow Can_{TG}T^*P$ of the connection space into the space of generalized canonical forms defined as follows
\be\label{idcor}\iota_\alpha(\alpha'):=\varphi\circ T\pi^*(\varphi)+\varphi\circ T\kappa_{\pi^*(\varphi)}(e)\circ(\alpha'_{\pi^*(\varphi)}-\alpha_{\pi^*(\varphi)})\circ T\pi^*(\varphi).\ee
The symplectic form $d\iota_\alpha(\alpha')$ is the pullback $L^*_{\id_{TP}+B}\omega_0$ of the standard symplectic form $\omega_0$ by the bundle morphism $(\id_{TP}+B)^*:T^*P\to T^*P$, where $\id_{TP}+B\in Aut_N TP$ is defined in (\ref{Bp}).
\end{cor}

Although there is no distinguished connection on a  principal $G$-bundle in general, such connections exist  in some particular cases. For example,  if the principal bundle $P(M,G)$ is trivial, $P=M\times G$, or when $P$ is a Lie group and $G$ is its subgroup. In the last case there exists the connection which is invariant with respect to the left action of $P$ on itself  and this connection is defined in a unique way, see Theorem 11.1 in \cite{Koba}.  Let us mention that  in the case when the reference connection  $\alpha$ is determined by some additional conditions, for example by the symmetry properties, then the connection $\alpha'$, see (\ref{idcor}),  can be naturally interpreted as an external field which interacts with a particle localized in the configuration space $P$.

In \cite{Mont,Sternberg} a $G$-equivariant diffeomorphism
$ I_\alpha:T^*P\stackrel{\sim}{\to} \ol P\times T^*_eG$
dependent on a fixed connection $\alpha$ was considered, where 
$$ \ol P:=\{(\tilde\varphi,p)\in T^*M\times P:\ \tilde\pi^*(\tilde\varphi)=\mu(p)\}$$
is the total space of the principal bundle $\ol P(T^*M,G)$ being the pullback of the principalbundle $P(M,G)$ to $T^*M$ by the projection $\tilde\pi^*:T^*M\to M$ of $T^*M$ on the base $M$. This diffeomorphism  is defined as follows
		$$\label{Ialpha} I_\alpha(\varphi):=({\Gamma}^*_\alpha(\pi^*(\varphi))(\varphi),\pi^*(\varphi), J_0(\varphi)).$$
		The correctness of the above definition follows from $\tilde\pi^*\circ \Gamma^*_\alpha=\mu\circ\pi^* $. The map $I_\alpha^{-1}:\ol P\times T^*_eG\to T^*P$ given by 
		\be\label{Ialphainv} I_\alpha^{-1}(\tilde\varphi,p,\chi)=\tilde\varphi\circ T\mu(p)+\chi\circ\alpha_p\ee 
	is	the inverse to $I_\alpha$.

		The natural right action of $Aut_{TG}TP$ on $T^*P$, defined for  $A\in Aut_{TG}TP$ by $(A^*\varphi)(\pi^*(\varphi)):=\varphi\circ A(\pi^*(\varphi))$, and the action of $G$ on $T^*P$ defined in (\ref{62}) transported by $I_\alpha$ to $\ol P\times T^*_eG$ are given by
			\be\label{Lambda} \Lambda_\alpha(A)(\tilde\varphi,p,\chi):=(I_\alpha\circ A^*\circ I_\alpha^{-1})(\tilde\varphi,p,\chi)=\ee
			$$=((\tilde\varphi\circ T\mu(p)+\chi\circ \alpha_p)\circ A(p)\circ \Gamma_\alpha(p),p,\chi)$$
			and by
			\be \label{86a}\psi^*_g(\tilde\varphi,p,\chi):=(I_\alpha\circ \phi^*_g\circ I_\alpha^{-1})(\tilde\varphi,p,\chi)=(\tilde\varphi,pg,Ad_{g^{-1}}^*\chi),\ee
			respectively. Setting $A=\id_{TP}+B$  or $A=\sigma_\alpha(\tilde A)$ in  (\ref{Lambda}) we obtain
		\be\label{Lambda1} \Lambda_\alpha(\id_{TP}+B)(\tilde\varphi,p,\chi)=(\tilde\varphi+\chi\circ\alpha_p\circ B(p)\circ \Gamma_\alpha(p),p,\chi)\ee
		or
		\be\label{Lambda2} \Lambda_\alpha(\sigma_\alpha(\tilde A))(\tilde\varphi,p,\chi)=(\tilde\varphi\circ\tilde A,p,\chi),\ee
		respectively.
		Summarizing, let us mention some  properties of the above two actions:\ben
		\item[(i)] The action  $\Lambda_\alpha$ of $Aut_{TG}TP$ on $ \ol P\times T^*_eG$ is  reduced to an action of $Aut_{TG}TP$ on $T^*M$ which preserves  the cotangent spaces $T^*_mM$, $m\in M$, and is realized on them by affine maps, see (\ref{Lambda}), (\ref{Lambda1}) and (\ref{Lambda2}).
		\item[(ii)] The action (\ref{86a}) of $G$  does not change $\tilde \varphi$ and commute with the action (\ref{Lambda}) of the group $Aut_{TG}TP$.
		\een

	Using $I_\alpha^{-1}:\ol P\times T_e^*G\to T^*P$ we pull  the generalized canonical form $\Theta(A)$ back to  $\ol P\times T_e^*G$. For this reason note that a vector $\xi_{(\tilde\varphi,p,\chi)}\in T_{(\tilde\varphi,p,\chi)}(T^*M\times P\times T_e^*G)$ is tangent to $\ol P\times T_e^*G\subset T^*M\times P\times T_e^*G$ if and only if 
	\be\label{89} T(\tilde\pi^*\circ pr_1)(\tilde\varphi,p,\chi)\xi_{(\tilde\varphi,p,\chi)}=T(\mu\circ pr_2)(\tilde\varphi,p,\chi)\xi_{(\tilde\varphi,p,\chi)},\ee
	where $pr_1(\tilde\varphi,p,\chi):=\tilde\varphi$ and $pr_2(\tilde\varphi,p,\chi):=p$. The equality (\ref{89}) follows from $\tilde\pi^*\circ pr_1=\mu\circ pr_2$. 	For $A=(\id_{TP}+B) \sigma_\alpha(\tilde A)$ we have
	\be\label{Icos} \langle (I_\alpha^{-1})^*\Theta(A)(\tilde\varphi,p,\chi),\xi_{(\tilde\varphi,p,\chi)}\rangle=\ee
	$$=\langle I_\alpha^{-1}(\tilde\varphi,p,\chi),A(\pi^*(I_\alpha^{-1}(\tilde\varphi,p,\chi)))\circ T\pi^*(I_\alpha^{-1}(\tilde\varphi,p,\chi))\circ TI_\alpha^{-1}(\tilde\varphi,p,\chi)\xi_{(\tilde\varphi,p,\chi)}\rangle=$$
	$$=\langle I_\alpha^{-1}(\tilde\varphi,p,\chi),A(\pi^*\circ I_\alpha^{-1})(\tilde\varphi,p,\chi)\circ T(\pi^*\circ I_\alpha^{-1})(\tilde\varphi,p,\chi)\xi_{(\tilde\varphi,p,\chi)}\rangle=$$
	$$=\langle \tilde\varphi\circ T\mu(p)+\chi\circ\alpha_p,A(p)Tpr_2(\tilde\varphi,p,\chi)\xi_{(\tilde\varphi,p,\chi)}\rangle=$$
		$$=\langle \tilde\varphi\circ \tilde A(\mu(p)),T\mu(p)\circ Tpr_2(\tilde\varphi,p,\chi)\xi_{(\tilde\varphi,p,\chi)}\rangle+
		\langle \chi\circ\alpha_p,A(p)Tpr_2(\tilde\varphi,p,\chi)\xi_{(\tilde\varphi,p,\chi)}\rangle=$$
		$$=\langle \tilde\varphi\circ \tilde A(\mu(p)),T(\tilde\pi^*\circ pr_1)(\tilde\varphi,p,\chi)\xi_{(\tilde\varphi,p,\chi)}\rangle+
		\langle \chi\circ\alpha_p,A(p)\circ Tpr_2(\tilde\varphi,p,\chi)\xi_{(\tilde\varphi,p,\chi)}\rangle,$$
		where we have used (\ref{Ialphainv}), (\ref{89}) and $\pi^*\circ I_\alpha^{-1}=pr_2$. Omitting $\xi_{(\tilde\varphi,p,\chi)}$  in (\ref{Icos}) we obtain
		\be\label{86}  (I_\alpha^{-1})^*\Theta(A)(\tilde\varphi,p,\chi)=\ee
		$$=\tilde\varphi\circ \tilde A(\mu(p))\circ T(\tilde\pi^*\circ pr_1)(\tilde\varphi,p,\chi)+\chi\circ\alpha_p\circ A(p)\circ Tpr_2(\tilde\varphi,p,\chi)=$$
		$$=pr_1^*(\tilde\Theta(\tilde A)(\tilde\varphi,p,\chi)+\langle pr_3(\tilde\varphi,p,\chi),pr_2^*(\Phi_{A^{-1}}(\alpha))(\tilde\varphi,p,\chi)\rangle,$$
		where $pr_3(\tilde\varphi,p,\chi):=\chi$.		The  symplectic form corresponding to  (\ref{86}) is given by 
		\be\label{sympl} d((I_\alpha^{-1})^*\Theta(A))= \ee
		$$=pr_1^*(d\tilde\Theta(\tilde A))+\langle d\ pr_3\ \stackrel{\wedge}{,}\  pr_2^*(\Phi_{A^{-1}}(\alpha))\rangle+\langle  pr_3, pr_2^*(d\Phi_{A^{-1}}(\alpha))\rangle.$$
		Let us note that $(I_\alpha^{-1})^*\Theta(A)$ consists of the pull back on $\ol P\times T^*_eG$ of the generalized canonical  form $\tilde \Theta(\tilde A)\in CanT^*M$ by $pr_1:\ol P\times T^*_eG\to T^*M$ and the part defined by the connection form  $\Phi_{A^{-1}}(\alpha)$.
		
		\bigskip
		
		\section{The Marsden-Weinstein reduction}\label{sec5}
		Considering $P$ as the configuration space of a physical system which has a symmetry described by $G$ one consequently assumes that its Hamiltonian $H\in C^\infty(T^*P)$ is a $G$-invariant function on $T^*P$, i.e. $H\circ \phi^*_g=H$ for $g\in G$. 
			Hence  it is   natural  to consider the class of Hamiltonian  systems on $G$-symplectic manifold $(T^*P, \omega_{A },J_0)$ with a $G$-invariant Hamiltonians $H$. 
			
			Using the isomorphism  $(T^*P, \omega_{A },J_0)\cong (\ol P\times T^*_eG, (I_\alpha^{-1})^*\omega_{A }, pr_3)$ of  $G$-symplectic manifolds, where the symplectic form $(I_\alpha^{-1})^*\omega_{A }$ is presented in  (\ref{sympl}) and the momentum map is $J_0\circ I_\alpha=pr_3$, one defines (see \cite{Mont, Sternberg}) the $G$-invariant Hamiltonian $H\in C^\infty(\ol P\times T^*_eG)$ as follows
		$$ H(\tilde\varphi,p,\chi):=(\tilde H\circ \ol\mu)(\tilde\varphi,p,\chi)+(C\circ pr_3)(\tilde\varphi,p,\chi),$$
where $\ol\mu:\ol P\to T^*M$ is the projection of the total space $\ol P$ of the principal $G$-bundle $\ol P(T^*M,G)$ on the base $T^*M$ and $\tilde H\in C^\infty(T^*M)$. Coming back to the phase space 	$(T^*P, \omega_{A },J_0)$	one obtains the  $G$-Hamiltonian system  with the Hamiltonian
$$		 H_{\alpha}(\varphi):=(H\circ I_\alpha)(\varphi)=(\tilde H\circ \Gamma^*_\alpha)(\varphi)+ (C\circ J_0)(\varphi).$$
Let us stress that in the case of $(T^*P,\omega_{A },J_0, H_\alpha)$ only the Hamiltonian $H_\alpha$ of the system  depends on $\alpha\in Conn P(M,G)$ and in the case of $(\ol P\times T^*_eG, (I_\alpha^{-1})^*\omega_{A}, pr_3, H)$  the only symplectic form  $(I_\alpha^{-1})^*\omega_{A}$.

In \cite{Kerner,Kummer, Mont, Sn,Sn2,Sou,Sternberg,Weinstein:1977,Wong} there were presented various models of the description of motion of a classical particle in the external Yang-Mills field given by the connection $\alpha$ and by the Hamiltonian $H$. In these models the basic symplectic structure on $T^*P$ is given by the standard symplectic form $\omega_0$.
 Here allowing the generalized symplectic form $\omega_{A }$ we extend the class of models under investigations. 

The $G$-invariance of the Hamiltonian system $(T^*P,\omega_{A },J_0,H_\alpha)$ allows ones to apply the Marsden-Weinstein reduction procedure \cite{MW}. For this reason we consider		
	 the dual  pair of Poisson manifolds 
 \unitlength=5mm
 \be\label{diagsymplpair}
 \begin{picture}(11,4.6)
    \put(4.5,4){\makebox(0,0){$(T^*P,\omega_{A })$}}
    \put(0,-1){\makebox(0,0){$(T^*P/G,\{\cdot,\cdot\}_{A /_G})$}}
    \put(9,-1){\makebox(0,0){$ (T^*_eG,\{\cdot,\cdot\}_{L-P})$}}
    \put(4,3.7){\vector(-1,-1){4}}
     \put(5,3.7){\vector(1,-1){4}}
       \put(1,2){\makebox(0,0){$\pi^*_G$}}
    \put(8,2){\makebox(0,0){$ J_0$}}
     \end{picture}\ee		
		\bigskip
		\ \\
		in the sense of Subsection 9.3 in \cite{coste}.
		Recall that  the symplectic form $\omega_{A }$ is a $G$-invariant two-form. The Poisson bracket $\{f,g\}_{A /G}$ of $f,g\in C^\infty(T^*P/G)$ is defined by $\{f\circ \pi^*_G,g\circ  \pi^*_G\}_A$, where we identify $C^\infty(T^*P/G)$ with the Poisson subalgebra $C^\infty_G(T^*P)\subset C^\infty(T^*P)$ of $G$-invariant functions and $\{\cdot,\cdot\}_{A }$ is the Poisson bracket on $C^\infty(T^*P)$ defined by $\omega_{A }$. By $\{\cdot,\cdot\}_{L-P}$ we denoted Lie-Poisson bracket on the dual $T^*_eG$ of the Lie algebra $T_eG$.
		
		Note that surjective submersions in (\ref{diagsymplpair}) are Poisson maps and the Poisson subalgebras $(\pi^*_G)^*(C^\infty(T^*P/G))$ and $J_0^*(C^\infty(T^*_eG))$ are mutually polar. As a consequence of the above one obtains the one-to-one correspondence between the coadjoint orbits $\mathcal{O}\subset T^*_eG$ of $G$ and the symplectic leaves $\mathcal{S}\subset T^*P/G$ of the Poisson manifold $(T^*P/G,\{\cdot,\cdot\}_{A /G})$ which is defined as follows 
	$$ \mathcal{S}=\pi^*_G(J_0^{-1}(\mathcal{O})) \quad {\rm and }\quad \mathcal{O}=J_0({\pi^*_G}^{-1}(\mathcal{S})).$$
		
		Let us stress that the manifold structure of a  symplectic leaf $\mathcal{S}$   does not depend on the choise of $A \in Aut_{TG}TP$, but its symplectic structure $\omega_{A }^{\mathcal{S}}$ does. The action of $Aut_{TG}TP$ on $T^*P$ commutes with the action (\ref{62}) of $G$ on $T^*P$, so, it defines an action of $Aut_{TG}TP$ on the quotient manifold $T^*P/G$. By the definition of $Can_{TG}T^*P$, the momentum map $J_{A }$ for $\omega_{A }$ coincides with $J_0=J_A$. Thus we conclude that the action of $A' \in Aut_{TG}TP$ on $T^*P/G$ preserves the symplectic leaves $\mathcal{S}$  and  transforms their  symplectic forms in the following way $\omega_{A }^{\mathcal{S}}\to \omega_{A'A }^{\mathcal{S}}$.

		Since   $I_\alpha:T^*P\to \ol P\times T^*_eG$ is a  $G$-equivariant map it defines a diffeomorphism 
		$$\label{Iquo} [I_\alpha]: T^*P/G\to \ol P\times _{Ad^*_G}T^*_eG$$
		of the quotient manifolds which transports the Poisson structure $\{\cdot,\cdot\}_{A/G}$ of $T^*P/G$  on the total space $ \ol P\times _{Ad^*_G}T^*_eG$ of the vector bundle  $\ol P\times _{Ad^*_G}T^*_eG\to T^*M$ over the symplectic manifold $(T^*M, d\tilde\Theta(\tilde A))$. Using (\ref{Iquo}) one obtains the isomorphisms
		$[I_{\alpha,\mathcal{O}}]=\pi^*_G(J^{-1}_0(\mathcal{O}))\stackrel{\sim}{\to} \ol P\times _{Ad^*_G}\mathcal{O}$ of symplectic leaves. If $A=\sigma_\alpha(\tilde\id_{TM})=\id_{TP}$ one obtains the diffeomorphisms of symplectic leaves constructed in \cite{Mont, Sternberg} where the coadjoint orbit $\mathcal{O}$ is the phase space for inner degrees of freedom. In this case  the symplectic manifold $(T^*M, d\tilde\gamma_0)$ is the phase space for external degrees of freedom and $\ol P\times _{Ad^*_G}\mathcal{O}$ is the total phase space of a classical particle interacted with Yang-Mills field described by $\alpha$ which was constructed in \cite{Sternberg}.

		If $\rho\in T^*_eG$ is such that $Ad^*_G\rho=\rho$ then $\mathcal{O}=\{\rho\}$. Hence the symplectic leaf $\S=\ol P\times Ad^*_G\mathcal{O}$ is isomorphic as a manifold with $T^*M$ but the reduced symplectic form $\omega ^S_{A }$ of $\S$ depends on the choice of $A \in Aut_{TG}TP$. For example the above situation happens for all $\rho\in T^*_eG$ if $G$ is a commutative group or if $\rho=0$.
		
		Ending let us mention that all constructions presented above have an equivariance properties with respect to the group $Aut_{TG}TP$.

\end{document}